\documentclass[12pt]{amsart}

\usepackage{mathrsfs}
\usepackage{amssymb}
\usepackage{cite}
\usepackage[all]{xy}
\usepackage{graphicx}
\usepackage{textcomp}
\usepackage{charter}
\usepackage[vflt]{floatflt}
\usepackage{mathtools}
\usepackage{enumerate}
\usepackage{wasysym}
\usepackage{rotating}
\usepackage{pdflscape}
\usepackage{fullpage}
\usepackage{bm}

\usepackage[pdftex,colorlinks]{hyperref}
\hypersetup{
pdfauthor={Shamovich, E.},
bookmarksnumbered,
pdfstartview={FitH}}%

\newtheorem{thm}{Theorem}[section]
\newtheorem{prop}[thm]{Proposition}
\newtheorem{lem}[thm]{Lemma}
\newtheorem{cor}[thm]{Corollary}
\newtheorem*{thm*}{Theorem}
\newtheorem*{cor*}{Corollary}

\theoremstyle{definition}

\newtheorem{example}[thm]{Example}
\newtheorem{rem}[thm]{Remark}

\theoremstyle{remark}

\newcommand{\Fix}{\operatorname{Fix}}

\numberwithin{equation}{section}

\usepackage{mathrsfs}

\newcommand{\GL}{\operatorname{{\mathbf GL}}}

\newcommand{\tr}{\operatorname{tr}}

\newcommand{\cF}{{\mathcal F}}

\newcommand{\cH}{{\mathcal H}}
\newcommand{\cI}{{\mathcal I}}

\newcommand{\cL}{{\mathcal L}}
\newcommand{\cM}{{\mathcal M}}

\newcommand{\cS}{{\mathcal S}}

\newcommand{\fB}{{\mathfrak B}}

\newcommand{\fI}{{\mathfrak I}}

\newcommand{\fV}{{\mathfrak V}}
\newcommand{\fW}{{\mathfrak W}}

\newcommand{\B}{{\mathbb B}}

\newcommand{\C}{{\mathbb C}}
\newcommand{\R}{{\mathbb R}}

\newcommand{\N}{{\mathbb N}}

\newcommand{\M}{{\mathbb M}}

\newcommand{\W}{{\mathbb W}}
\newcommand{\D}{{\mathbb D}}

\newcommand{\balpha}{\bm{\alpha}}
\newcommand{\bbeta}{\bm{\beta}}
\newcommand{\bw}{\mathbf{w}}

\def\mcc{M\raise.5ex\hbox{c}C}
\def\mccarthy{M\raise.5ex\hbox{c}Carthy}

\begin{document}

\title{On fixed points of self maps of the free ball}

\author{Eli Shamovich}
\address{Dept. of Pure Mathematics, University of Waterloo, Waterloo, ON, Canada}
\email{eshamovi@uwaterloo.ca}


\thanks{}

\begin{abstract}
In this paper, we study the structure of the fixed point sets of noncommutative self maps of the free ball. We show that for such a map that fixes the origin the fixed point set on every level is the intersection of the ball with a linear subspace. We provide an application for the completely isometric isomorphism problem of multiplier algebras of noncommutative complete Pick spaces.
\end{abstract}

\maketitle


\section{Introduction} \label{sec:intro}

Function theory and hyperbolic geometry of $\B_d$, the unit ball of $\C^d$, were studied extensively throughout the years, see for example \cite{Rudin} and \cite{Gar07}. The fact that $\B_d$ is the unit ball of a finite dimensional Hilbert space leads to a generalization of many classical results from the unit disc setting, like  the Schwarz lemma and the Julia-Caratheodory-Wolff theorem (see \cite{ERS08}, \cite{GoeRei84}, \cite{Her63}, \cite{ReiSho05}, and \cite{Rudin} for details). 

In operator algebra theory the Drury-Arveson space $\cH^2_d$, the model space for commuting row contractions, is a reproducing kernel Hilbert space of analytic function on $\B_d$ with reproducing kernel $k_d(z,w) = \frac{1}{1 - \langle z, w \rangle}$ (see \cite{Arv98} and \cite{Drury}). This space is a complete Pick space, i.e., the multipliers of the Drury-Arveson space admit an interpolation theorem for matrix valued functions generalizing the classical Nevanlina-Pick interpolation theorem in the unit disc. Let us write $\cM_d$ for the algebra of multipliers on $\cH_d^2$, it is a maximal abelian WOT-closed operator subalgebra of $B(\cH_d^2)$ generated by the operators $M_{z_j}$ of multiplication by coordinate functions. In \cite[Theorem 8.2]{AM_Book} it was shown that the Drury-Arveson space for $d = \infty$ is the universal complete Pick space, namely, if $\cH$ is a separable complete Pick reproducing kernel Hilbert space on a set $X$ with kernel $k$, then there exists an embedding $b \colon X \to \B_{\infty}$ and a nowhere vanishing function $\delta$ on $X$, such that $k(x,y) = \overline{\delta(x)}\delta(y) k_{\infty}(b(x),b(y))$ and $\cH$ is isometrically embedded in $\delta \cH_d^2$. 

Let $V \subset \B_d$ be an analytic subvariety of $\B_d$ cut out by functions in $\cM_d$. We can associate to it a reproducing kernel Hilbert space $\cH_V$ spanned by kernel functions $k_d(\cdot,w)$ for $w \in V$. This space turns out to be a complete Pick space and the multiplier algebra $\cM_V$ of $\cH_V$ is completely isometrically isomorphic to $\cM_d/\cI_V$, where $\cI_V$ is the WOT-closed ideal of functions vanishing on $V$. It is thus natural to ask to what extent does the algebra $\cM_V$ determine the variety $V$ and vice versa. The isomorphism problem for subvarieties of $\B_d$ cut out by multipliers of the Drury-Arveson space was studied by Davidson, Ramsey and Shalit. In \cite{DRS11} and \cite{DRS15} they studied the algebra $\cM_V$ and its norm closed analog and proved that if $V, W \subset \fB_d$ are subvarieties of $\fB_d$, such that their affine span is all of $\C^d$, then $\cM_V$ is completely isometrically isomorphic to $\cM_W$ if and only if there is an automorphism of $\B_d$ mapping $V$ onto $W$ (see also \cite{SalomonShalit} for a survey and more results on the commutative isomorphism problem). One of the main tools in the proof of the theorem is a theorem that appears both in \cite{Rudin} and \cite{Her63} and states that the fixed point set of a self map of $\B_d$ is the intersection of $\B_d$ with an affine subspace (see also \cite[Theorem 23.2]{GoeRei84} and \cite[Theorem 6.3]{KRS01}).

The noncommutative (nc for short), or free functions were introduced by Taylor in \cite{Tay72frame} and \cite{Tay73}. Taylor's goal was to facilitate noncommutative functional calculus and thus he endeavored  to give topological algebras analogous to the classical Frechet algebras of analytic functions on open domains in $\C^d$. Voiculescu in \cite{Voic85}, \cite{Voic86}, \cite{Voic04} and \cite{Voic10} developed the ideas of Taylor in the context of free probability. Helton, Klep , McCullough and Schweighofer applied noncommutative analysis in order to obtain dimension free relaxation of the LMI containment problem (see \cite{HKM13-relax} and \cite{HKMS15}). Their results were extended and improved upon by Davidson, Dor-On, Shalit and Solel in \cite{DDSS17},  Passer, Shalit and Solel in \cite{PSS17} and Fritz, Netzer and Thom in \cite{FNT17}. For other applications to free real algebraic geometry see for example \cite{HM2004} and \cite{HM2012}. More applications to free probability were provided by Belinschi, Popa and Vinnikov in \cite{BPV13} and Popa and Vinnikov in \cite{PopVin13}. A fundamental book \cite{KVBook} on the properties of nc function was written by Kaliuzhnyi-Verbovetskyi and Vinnikov. The theory of bounded functions on free domains was studied by Agler and \mccarthy{} in \cite{AM00}, \cite{AM15}, \cite{AM15a}, \cite{AM15b}, \cite{AM15c} and \cite{AM15d}. They have obtained interpolation and realization results with applications to $H^{\infty}$ functional calculus on free analogs of polynomial polyhedra. Similar interpolation and realization results were obtained by Ball, Marx and Vinnikov in \cite{BMV15a} and \cite{BMV15b}, where they have also developed the theory of noncommutative reproducing kernels and defined the complete Pick property for such kernels. Muhly and Solel formulated a much more general theory using $W^*$-correspondences in \cite{MS08} (see also \cite{MS13}).

In Section \ref{sec:main} we will prove an analog of the result of Rudin and Herv\'e for the free ball. Automorphisms of the ball were studied by Davidson and Pitts in \cite{DavPitts2}, \mccarthy{} and Timoney in \cite{McT16}, and Popescu in \cite{Popescu10}. More generally self maps of the free ball and quantizations of Cartan domains of type I were studied by Helton, Klep, McCullough and Slinglend in \cite{HKMS09}. They, however, have studied self maps with nice boundary properties and here we make no such assumptions. Automorphisms of quantizations of Cartan domains of type I were also studied in \cite{McT16}. A study of isomorphisms of free LMI domains was carried out by Helton, Klep and McCullough in \cite{HKM11} and by Augat, Helton, Klep and McCullough in \cite{AHKM16}. In our study of self nc-maps of the free ball, we employ the classical techniques of complex geodesics developed by Vesentini in \cite{Ves81} and \cite{Ves82}, Vigu\'e in \cite{Vig84} and \cite{Vig85} and others.

Analogously to the commutative case, the free matrix ball admits a noncommutative reproducing Szego kernel and the associated nc reproducing kernel space is the full Fock space. The fact that the full Fock space is the noncommutative analog of the Drury-Arveson space was observed by Bunce \cite{Bunce}, Frazho \cite{Frazho} and Popescu \cite{Popescu89}, \cite{Popescu91}, \cite{Popescu92} and \cite{Popescu95}. Noncommutative bounded functions on the free ball and its hyperbolic geometry were studied by Popescu in \cite{Popescu06a}, \cite{Popescu06b}, \cite{Popescu09}, \cite{popescu09II}, \cite{Popescu10} and \cite{Popescu10II} and Davidson and Pitts in \cite{DavPitts1} and \cite{DavPitts2}. Noncommutative Nevanlinna-Pick interpolation on the ball was obtained by Arias and Popescu in \cite{AriasPopescu} and Davidson and Pitts in \cite{DavPittsPick} using operator algebraic methods. A more generalHardy algebras in the setting of $W^*$-correspondences were considered by Muhly and Solel in \cite{MS04} and \cite{MS11}. They have also obtained analogs of Nevanlinna-Pick results in this setting.

The noncommutative analog of the isomorphism problem for multiplier algebras of subvarieties of the free ball was formulated by Salomon, Shalit and Shamovich in \cite{3S}, where the completely isometric isomorphism case was studied. In \cite{3S} a generalization of the results of \cite{DRS11} and \cite{DRS15} for homogeneous varieties was provided and it was shown that in the nonhomogeonous case two varieties with completely isometrically isomorphic multiplier algebras are biholomorphic. If $\fV \subset \fB_d$ and $\fW \subset \fB_e$ are two subvarieties with completely isometrically isomorphic multiplier algebras, then there exist two nc maps $f \colon \fB_d \to \fB_e$ and $g \colon \fB_e \to \fB_d$, such that $f \circ g|_{\fW} = \operatorname{id}_{\fW}$ and $g\circ f|_{\fV} =\operatorname{id}_{\fV}$. IF $\fV$ and $\fW$ are homogeneous then \cite[Theorem 8.4]{3S} shows that there exists a $k$ and an automorphism$\varphi$ of $\fB_k$ , such that $\fV, \fW \subset \fB_k$ and $\varphi$ maps $\fV$ onto $\fW$. In the commutative case, as stated above, Davidson, Ramsey and Shalit proved that for completely isometric isomorphism can be realized by an automorphism of the ball without the assumption that the varieties are homogeneous. In Section \ref{sec:application} we will provide as an application of our main results a noncommutative generalization of the result of Davidson, Ramsey and Shalit on the fact that a completely isometric isomorphism of multiplier algebras of two varieties with scalar points, that are embedded in a non-degenerate way in the free ball is implemented via an automorphism of the free ball.

\section{Basic Definitions and Notations} \label{sec:notation}

From this point on $d$ is a positive integer. Let us denote by $\M_d = \sqcup_{n=1}^{\infty} M_n(\C)^{\oplus d}$ and we think of this space as the space of $d$-tuples of matrices of all sizes. Let us write $\fB_d$ for the nc-ball in $\M_d$, namely the set of all $d$-tuples of matrices $X = (X_1,\ldots,X_d)$, such that $\sum_{j=1}^d X_j X_j^* < I$. In other words, $\fB_d(n)$ is the set of all strict $d$-row contractions of size $n \times n$. In fact $\fB_d(n)$ is a Cartan domain of type I for every $n$. To each point $X \in \fB_d(n)$ we associate a completely positive map $\Phi_X(T) = \sum_{j=1}^n X_j T X_j^*$. We will say that a point $X \in \M_d(n)$ is generic if the algebra generated by $X_1,\ldots,X_d$ is all of $M_n(\C)$. There is a natural action of $\GL_n(\C)$ on $\M_d(n)$, given by $S \cdot X = S^{-1} X S = (S^{-1} X_1 S, \ldots, S^{-1} X_d S)$, for all $X \in \M_d(n)$ and $S \in \GL_n(\C)$.

The main objects of study in this paper are nc functions and maps. In this section, we will provide some basic definitions. For properties of nc functions the reader is referred to the works of Agler and \mccarthy{} (cf. \cite{AM15}) and the foundational book \cite{KVBook}. 

By a direct sum of points $X \in \M_d(n)$ and $Y \in \M_d(m)$, we mean $X \oplus Y = \begin{pmatrix} X & 0 \\ 0 & Y \end{pmatrix}$. Let $\Omega \subset \M_d$ be a set closed under direct sums. A function $f \colon \Omega \to \M_1$ is said to be an nc function if the following conditions hold:
\begin{itemize}
\item $f$ is graded, namely $f(\Omega(n)) \subset \M_1(n)$, for every $n \geq 1$;

\item $f$ respects direct sums, i.e., if $X \in \Omega(n)$ and $Y \in \Omega(m)$, then $f(X\oplus Y) = f(X) \oplus f(Y)$;

\item $f$ respects similarities, i.e., if $X \in \Omega(n)$ and $S \in \GL_n(\C)$, such that $S \cdot X = S^{-1} X S \in \Omega(n)$, then $f(S^{-1} X S) = S^{-1} f(X) S$.
\end{itemize}
One can replace the second and third item by a single item that states that $f$ respects intertwiners, but we will not use this property.

If we assume additionally that $\Omega(n)$ is open for every $n$ and we have a point $P = \begin{pmatrix} X & Z \\ 0 & Y \end{pmatrix}$, with $X \in \Omega(n)$ and $Y \in \Omega(m)$, then
\[
f(P) = \begin{pmatrix} f(X) & \Delta f(X,Y)(Z) \\ 0 & f(Y) \end{pmatrix}.
\]
Here $\Delta f(X,Y)$ is the noncommutative difference-differential operator and it is a linear map $\Delta f(X,Y) \colon M_{n,m}(\C)^{\oplus d} \to M_{n,m}(\C)$. The properties of the difference-differential operator are studied thoroughly in \cite{KVBook}. Using this property one gets a rather surprising result that a mere local boundedness assumption ensures analyticity of $f$ on every level as a function of the coordinates of the matrices. Since we will deal with nc-maps from $\fB_d$ to itself, they are thus automatically analytic. Furthermore, for every $X \in \Omega(n)$, $\Delta f(X,X)$ is, in fact, the derivative of $f$ at $X$, when we consider $f$ as an analytic map $f \colon \Omega(n) \to M_n$. The nc-difference differential operator has the following property that will be used several times throughout the paper:
\[
\Delta f(X\oplus Y, X\oplus Y) \left( \begin{pmatrix} P_{11} & P_{12} \\ P_{21} & P_{22} \end{pmatrix}\right) = \begin{pmatrix} \Delta f(X,X)(P_{11}) & \Delta f(X,Y)(P_{12}) \\ \Delta f(Y,X)(P_{21}) & \Delta(Y,Y)(P_{22}) \end{pmatrix}.
\]
In particular, if $\Omega(1) \neq \emptyset$, then for every vector $\balpha \in \Omega(1)$ we have: that:
\[
\Delta f( \balpha^{\oplus n}, \balpha^{\oplus n}(\left( \left\{ \bw_{ij} \right\}_{i,j = 1}^n \right) = \left\{ \Delta f (\balpha, \balpha)(\bw_{ij}) \right\}_{i,j = 1}^n.
\]
Namely, the derivative of $f$ at the point $\balpha^{\oplus n}$ is the ampliation of the derivative of $f$ at $\balpha$, i.e., $\Delta f(\balpha^{\oplus n}, \balpha^{\oplus n}) = \Delta f(\balpha, \balpha) \otimes I_{M_n}$.

A nc-map $f \colon \Omega \to \M_d$ is just a map, such that every coordinate is an nc-function. If we have a self nc-map $f \colon \Omega \to \Omega$, we will write $\Fix(f)$ for the set of fixed points of $f$. Our main focus in this paper will be the set $\Fix(f)$ for a self nc-map of $\fB_d$ that fixes a scalar point.

The theory of bounded nc functions and nc maps on the ball has been studied in relation with operator algebras by Popescu in a series of papers (see for example \cite{Popescu89},\cite{Popescu91},\cite{Popescu06a},\cite{Popescu06b},\cite{Popescu10}), Arias and Popescu in \cite{AriasPopescu}, Davidson and Pitts in \cite{DavPitts1}, \cite{DavPitts2} and \cite{DavPittsPick} and Salomon, Shalit and the author in \cite{3S}. We will provide more details in Section \ref{sec:application}, where we describe an application to the study of operator algebras arising as multipliers of certain noncommutative complete Pick spaces.  

We will also need some results from complex hyperbolic geometry of convex domains in $\C^d$. Here we will briefly recall some basic definitions and results from the theory, for more information, the reader is referred to the excellent books \cite{Aba89}, \cite{FrVe80}, \cite{GoeRei84}, and \cite{Kob98}. Let us denote by $\D$ the unit disc in the complex plane. We will equip $\D$ with the Poincare-Bergman metric, $\rho$, that will make $\D$ a complex hyperbolic space. It is a consequence of the Schwarz-Pick lemma that the isometries of the disc are precisely the holomorphic automorphisms of the disc. Furthermore, every self map of the disc is a contraction with respect to the Poincare-Bergman metric. 

\newcommand{\hol}{\operatorname{Hol}}

Let $U \subset \C^d$ be a bounded domain. Let $V \subset \C^{d^{\prime}}$ be another domain and let us write $\hol(U,V)$ for the set of holomorphic maps from $U$ to $V$. One can define in general many metrics and pseudometrics invariant under the holomorphic automorphisms of $U$. Two such are the Caratheodory and Kobayashi pseudometrics. One defines the Caratheodory pseudometric by
\[
c(z,w) = \sup \{ \rho(f(z),f(w)) \mid f \in \hol(U,\D)\}.
\]
The Kobayashi pseudometric has a bit more complicated definition, that we do not present here, since for bounded convex domains in $\C^n$ those pseudometric are in fact metrics (this is true even if we omit convexity, see for example \cite[Theorem 2.3.14]{Aba89}) and they coincide by a result of Lempert \cite{Lem82}. Since we will only consider the free unit ball we will talk about the Caratheodory metric and assume from now on that $U$ is convex. A geodesic between two points $w_1,w_2 \in U$ with respect to the Caratheodory metric is a holomorphic map $\varphi \colon \D \to U$, such that there exist $z_1,z_2 \in \D$, such that $\varphi(z_1) = w_1$ and $\varphi(z_2) = w_2$ and $\varphi$ is isometric as a map from the unit disc with the Poincare-Bergman metric to $U$ with the Caratheodory metric. If $U$, for example, is the unit ball of some norm $\|\cdot\|$ on $\C^d$, then for every $w \in U$, $c(0,w) = \rho(0,\|w\|)$ and the map $z \mapsto z w$ is a complex geodesic. As in the case of the disc every $f \in \hol(U,U)$ is a contraction with respect to the Caratheodory metric and a holomorphic automorphism of $U$ is an isometry.

In fact the Caratheodry metric is an integrated form of an infinitesimal norm $\gamma(w,v)$ on the tangent space $T_w U$ to $U$ at $w$. One can show that a map $\varphi \in \hol(\D,U)$ is a complex geodesic if and only if there exist two points $z_1,z_2 \in \D$, such that $c(\varphi(z_1),\varphi(z_2)) = \rho(z_1,z_2)$. Similarly, $\varphi$ is a geodesic if there exists $z \in \D$, such that the map $d\varphi \colon T_z \D \to T_{\varphi(z)} U$ is an isometry with respect to the induced norms.

Vigu\'e has shown in \cite{Vig84} and \cite{Vig85} that if we have $f \in \hol(U,U)$, then every two distinct fixed points of $f$ are connected by a complex geodesic consisting  entirely of fixed points of $f$. Alternatively, if $w \in U$ is fixed by $f$ and $v \in T_w U$ is fixed by $df$, then there is a complex geodesic fixed by $f$ through $w$ that is tangent to $v$ at $w$. We will use those results in the next section.

\section{Main Result} \label{sec:main}

A classical result of Rudin \cite[Theorem 8.2.2]{Rudin} and Herv\'e \cite[Theorem 1]{Her63} (see also \cite{GoeRei84} and \cite{KRS01}) asserts that the fixed point set of a holomorphic self map of the unit ball in $\C^n$ is an intersection of an affine subspace with the unit ball. Since the unit ball is homogeneous we may apply an automorphism and assume that $g$ fixes the origin. If $f(0) = 0$, then the fixed point set is an intersection of the ball with a linear subspace and by applying a unitary we may assume that this subspace is defined by the vanishing of some coordinates. The goal of this section is to provide a free analog of this claim. Our first order of business is to show that if $f \colon \fB_d \to \fB_d$ is an nc-map, such that $f(0) = 0$ and $V(1)$ is the subspace of fixed points of $f$ on the first level, then $f$ fixes the points that satisfy the linear relations of $V$ on every level, i.e. if we set $V(n) = V(1) \otimes M_n$, then $V = \sqcup_{n=1}^{\infty} V(n) \subset \Fix(f)$.

\begin{lem} \label{lem:one-dim_fixed}
Let $f \colon \fB_d \to \fB_d$ be an nc-map, such that $f(0) = 0$. If $f$ fixes $(X,0,\ldots,0)$, for some $X \neq 0$, then $f$ fixes all points of the form $(Z,0,\ldots,0)$.
\end{lem}
\begin{proof}
Assume there exists a point $\lambda \neq 0$ in the spectrum of $X$. Applying similarities we can transform $X$ into an upper-triangular matrix with $\lambda$ one of the entries on the diagonal. Since $f$ is an nc-map, we conclude that $f(\lambda,0,\ldots,0) = (\lambda,0,\ldots,0)$. By \cite[Theorem 8.2.2]{Rudin} we know that all points of the form $(z,0,\ldots,0)$ are fixed. Taking direct sums and similarities we can conclude that every point of the form $(Z,0,\ldots,0)$ is fixed with $Z$ diagonalizable. However, since diagonalizable matrices are dense in the matrix algebra, we conclude that all points of the form $(Z,0,\ldots,0)$ are fixed.

If the spectrum of $X$ does not contain a non-zero point, then $X$ is nilpotent. Hence $X$ is similar to a direct sum of Jordan blocks. Since $f$ is nc, we may assume that $X$ is a single Jordan block. Since both $0$ and $(X,0\ldots,0)$ are fixed, we get that $\Delta f(0,0)$ fixes $(1,0,\ldots,0)$ and thus the disc $(z,0,\ldots,0)$ is fixed on the first level and we proceed as above.
\end{proof}

An immediate corollary is that if $f$ fixes the points $(z,0,\ldots,0)$, with $z \neq 0$ on the first level, then it fixes all of the points of the form $(Z,0,\ldots,0)$.

The following proposition is one of the main tools that we use to obtain the main result.

\begin{lem} \label{lem:similar_to_coisometry}
If $X \in \fB_d(n)$ is a generic point, then there exists $S \in \GL_n(\C)$, such that $(S^{-1} X S)(S^{-1} X S)^* = r I$, for some $0 < r < 1$.
\end{lem}
\begin{proof}

By \cite[Theorem 2]{Far96} since $X$ is generic the map $\Phi_X$ is irreducible, hence we can apply \cite[Theorem 2.3]{EH78} to find $r >0$ that is a simple eigenvalue of $\Phi_X$ and a positive-definite $A$, such that $\Phi_X(A) = r A$. Since $X$ is a strict contraction the spectral radius of $\Phi_X$ is strictly less than $1$ and in particular, $r < 1$. Set $S = \sqrt{A}$, the unique positive definite square root of $A$. Multiplying the equality $\Phi_X(A) = r A$ both on the left and on the right by $S^{-1}$ we get
\[
r I = S^{-1} \Phi_X(A) S^{-1} = S^{-1} X S S X^* S^{-1} = S^{-1} X S \left( S^{-1} X S \right)^*.
\]


\end{proof}

Note that since $\fB_d(n)$ is a ball of a norm, we know that for every $X \in \fB_d(n)$, there is a complex geodesic that connects $X$ to $0$, given by $z \mapsto z X/\|X\|$, where the norm is the row norm. By \cite[Theorem A]{Ara81} and \cite[Theorem 3.4]{Lew95} the coisometries, namely points $X \in \M_d(n)$ that satisfy $\Phi_X(I) = I$, are real exposed points of $\fB_d(n)$ and thus, in particular, are complex extreme points. Recall that for a convex domain $U \subset \C^m$, a point $x \in \partial U$ is said to be complex extreme if the only vector $y \in \C^m$ satisfying $x + \Delta y \subset \overline{U}$ is $y = 0$ (see \cite[Section 2.6]{Aba89}). Since not every point on the boundary is a coisometry, the geodesic described above is in general not unique, i.e., there might exist other complex geodesics that connect $0$ to $X$. However, if we have $X \in \fB_d(n)$, such that $X/\|X\|$ is a coisometry, then the above geodesic is unique. We will need the following result:

\begin{lem} \label{lem:der_fix_coisom}
Let $f\colon \fB_d(n) \to \fB_d(n)$ be a holomorphic map, such that $f(0) = 0$. Let $X \in \overline{\fB_d(n)}$ be a coisometry. If $\Delta f(0,0)(X) = X$. Then the complex geodesic $z \mapsto z X$ is fixed by $f$.
\end{lem}
\begin{proof}
The proof follows the idea of \cite[Theorem 8.2.2]{Rudin}. Since $X$ is an exposed point there exists a support hyperplane $H$ of $\fB_d(n)$, such that $H \cap \overline{\fB_d(n)} = \{X\}$. Normalizing we may assume that there is a linear functional $\varphi$ on $\M_d(n)$, such that $\varphi(X) = 1$ and $\Re \varphi(Y) < 1$, for every $Y \in \overline{\fB_d(n)} \setminus \{X\}$. Now we define a function on the unit disc $g(z) = \varphi(f(z X))$. Taking the derivative at the origin we see that $g^{\prime}(0) = \varphi(\Delta f(0,0)(X)) = 1$ and thus by the Schwarz lemma $g(z) = z$. We conclude that for every $0 < r <1$ we have $\frac{1}{r} \varphi(f(r X)) = 1$ and thus by our assumption on $\varphi$ we get that $f(r X) = r X$. To conclude the proof we note that if $X$ is a coisometry, then so is $e^{i t} X$, for every $t \in \R$. Repeating the argument above we see that for every $z \in \D$, $f(z X) = z X$.
\end{proof}

\newcommand{\Span}{\operatorname{Span}}
\newcommand{\matsp}{\operatorname{mat-span}}

Let $X \in \M_d(n)$ and denote by $\cL_X$ the space of all homogeneous linear polynomials on $d$ variables of degree one satisfied by the coordinates of $X$. Applying a unitary we can always assume that $\cL_X$ is spanned by $z_k,\ldots,z_d$. In \cite{3S} Salomon, Shalit and the author have defined the notion of a matrix span of a set of points $\cS \subset \M_d$. Recall that the matrix span of a set $\cS$ is defined via
\[
\matsp(\cS)(n) = \Span \left\{ \left(I_d \otimes T \right)(X) \mid X \in \cS(n),\, T \in \cL(M_n) \right\}.
\]
Here $\cL(M_n)$ stands for the linear operators on $M_n(\C)$ and for each $T \in \cL(M_n)$ we have $\left( I_d \otimes T\right)(X) = \left(T(X_1), \ldots, T(X_d)\right)$.

\begin{lem} \label{lem:mat-span}
If $\cS = \{X\}$ for $X \in \M_d(n)$, then $\matsp(\cS)(n)$ is precisely the set of points $Z \in M_d(n)$, such that $\cL_X \subset \cL_Z$.
\end{lem}
\begin{proof}
First, let us assume that $\cL_X = \{0\}$, i.e, the coordinates of $X$ are linearly independent. Then applying linear transformations $T \in \cL(M_n)$ coordinate-wise, we can get any point in $\M_d(n)$ and thus we are done. Now assume that $X = \left(X_1,\ldots,X_{k-1},0,\ldots,0 \right)$, with $X_1,\ldots,X_{k-1}$ linearly independent. As in the previous case we can get by applying linear transformations coordinate-wise any point of the form $Z = \left(Z_1,\ldots,Z_{k-1},0,\ldots,0\right)$. In particular, every such point satisfies $\cL_X \subset \cL_Z$. Applying a linear transformation on the coordinates just changes bases in the spaces of linear polynomials and thus the above statement is true for any $X$.

To prove the converse inclusion observe that if $Z \in \matsp(\cS)(n)$, then there exist a linear transformation $T \in \cL(M_n)$, such that 
\[
Z = \left( I_d \otimes T \right)(X) = \left(T(X_1),\ldots,T(X_d)\right).
\]
Hence if $p = \sum_{j=1}^d \alpha_j z_j \in \cL_X$, then by linearity $p(Z) = T(p(X)) = 0$ and thus $p \in \cL_Z$.

\end{proof}

\begin{prop} \label{prop:fixed_level_1}
If $f$ is an nc self-map of $\fB_d$, such that $f$ fixes the subspace of points of the form $(z_1,\ldots,z_{k-1},0,\ldots,0)$ in $\fB_d(1)$, for $k > 2$, then $f$ fixes all points of the form $(Z_1,\ldots,Z_{k-1},0,\ldots,0)$.
\end{prop}
\begin{proof}
By Lemma \ref{lem:one-dim_fixed} we know that the points $(Z_1,0,\ldots,0)$, $(0,\ldots,0,Z_{k-1},0,\ldots,0)$ are fixed by $f$. Now \cite[Lemma 8.1]{3S} implies that $\Delta f(0,0)$ fixes the matrix span of these points and thus fixes all points of the form $(Z_1,\ldots,Z_{k-1},0,\ldots,0)$. 
Consider the $n$-th level $\fB_d(n)$. Let $Z = (Z_1,\ldots,Z_{k-1},0,\ldots,0)$ be a point such that $Z/\|Z\|$ is coisometric. As was mentioned there is a unique geodesic passing between $0$ and $Z$ and it is just the disc through the two points. Additionally, the above argument shows that $\Delta f(0,0)(Z) = Z$ and it is the derivative of $f$ at the origin on level $n$. 
Applying Lemma \ref{lem:der_fix_coisom} we get that the geodesic $w \mapsto w Z$ is fixed by $f$. By Lemma \ref{lem:similar_to_coisometry} every generic point of the form $(X_1,\ldots,X_{k-1},0,\ldots,0)$ is similar to such a $Z$ we conclude that every generic point of this form is fixed. Since $k > 2$ the generic points are dense in the set of all points of this form and thus the entire set is fixed.
\end{proof}

This proposition combined with Lemma \ref{lem:one-dim_fixed} gives us the following corollary.
\begin{cor} \label{cor:fixed_subspace}
Let $f \colon \fB_d \to \fB_d$ be an nc-map, such that $f(0) = 0$. Let $V(1)$ be the fixed point set of $f$ on the first level. Set $V(n) \subset \fB_d(n)$ the subspace of matrices, such that their coordinates satisfy the linear relations of $V(1)$ and $V = \sqcup_{n = 1}^{\infty} V(n)$. Then $f$ fixes $V$.
\end{cor}

\begin{prop} \label{prop:fixes_lin_independent_point}
Let $f \colon \fB_d \to \fB_d$ be an nc-map, such that $f(0) = 0$. If there exists a generic point $X \in \fB_d(n)$, such that $f(X) = X$, then $f(Z) = Z$, for every $Z \in \fB_d$, such that $\cL_X \subset \cL_Z$. In particular, if the coordinates of $X$ are linearly independent, i.e., $\cL_X = \{0\}$, then $f$ is the identity map.
\end{prop}
\begin{proof}
Since $f$ is nc and it fixes $X$, it is forced to fix the similarity orbit of $X$ in the ball, namely $\left(\GL_n(\C) \cdot X \right) \cap \fB_d(n)$. By Lemma \ref{lem:similar_to_coisometry} the similarity orbit contains a point $Y$, such that $Y/\|Y\|$ is a coisometry. Note that $Y$ is also generic and $\cL_Y = \cL_X$, since the action of $\GL_n(\C)$ is linear. By \cite[Corollary 4.2]{Vig84} we have that there exists a complex geodesic consisting of fixed points of $f$ that connects $0$ to $Y$. Since $Y/\|Y\|$ is a coisometry and thus a complex extreme point on the boundary, this geodesic is necessarily $z \mapsto z Y/\|Y\|$. Thus, in particular, the derivative of $f$ at $0$ fixes this line as well. Since $f$ is an nc map, the derivative at $0$ is $\Delta f (0,0) \otimes I_{M_n}$ and it fixes the line $z Y$. By \cite[Lemma 8.1]{3S} we have that it must fix the intersection of the matrix span of $Y$ with the ball. As above we may assume that $Y = (Y_1,\ldots, Y_{k-1},0,\ldots,0)$, with $k > 2$ and $Y_1,\ldots, Y_{k-1}$ linearly independent. By Lemma \ref{lem:mat-span} the matrix span of $Y$ on level $n$ is precisely the set of all points of the form $(Z_1,\ldots,Z_{k-1},0,\ldots,0)$, i.e., $\cL_Y \subset \cL_Z$. Now as in the proof of Proposition \ref{prop:fixed_level_1}, since the $\Delta f(0,0)$ fixes the matrix span of $Y$. Applying again Lemmas \ref{lem:der_fix_coisom} and \ref{lem:similar_to_coisometry} and using the density of the generic points we obtain that all of the points of the form $\left(Z_1,\ldots,Z_{k-1},0,\ldots,0\right)$ are fixed by $f$.

In particular, if the coordinates of $X$ are linearly independent, then the coordinates of $Y$ are linearly independent and it implies that the derivative fixes $\fB_d(n)$. By the free version of Cartan's theorem (see \cite{McT16}, \cite{Popescu10} and \cite{3S}). We conclude that $f$ is the identity.
\end{proof}

\begin{rem}
The above lemma is the noncommutative analog of a claim from \cite{Rudin} that the fixed points of a holomorphic map $f \colon \B_d \to \B_d$, such that $f(0) = 0$ are precisely the fixed points of $D f(0)$, the derivative of $f$ at the origin.
\end{rem}

\begin{lem} \label{lem:jordan}
Let $f \colon \fB_d \to \fB_d$ be an nc-map, such that $f(0) = 0$. If $X \in \fB_d(n)$ is a fixed point of $f$, then the geometric and algebraic multiplicities of $1$ as an eigenvalue of $\Delta f(X,X)$ are the same.
\end{lem}
\begin{proof}
Consider the iterates of $f$, namely $f_1 = f$ and $f_k = f \circ f_{k-1}$. Since $X$ is a fixed point of $f$, it is a fixed point of every $f_k$. Taking the derivative of $f_k$ we find by the chain rule that $\Delta f_k(X,X) = \Delta f(X,X)^k$. Since each $f_k$ is a self map of $\fB_d(n)$, by \cite[Proposition V.1.2]{FrVe80} the differential $\Delta f_k(X,X)$ is contractive with respect to the infinitesimal Caratheodory metric on the tangent space at $X$. By \cite[Corollary 5.8]{Ves82} the infinitesimal Caratheodory metric is equivalent to the row norm on $M_n^d$. Thus there exists a constant $C > 0$, such that for every $k$, $\| \Delta f_k(X,X)\|_{2,2} = \|\Delta f(X,X)^k \|_{2,2} \leq C$, where $\|\cdot\|_{2,2}$ is the operator norm with respect to the Hilbert-Schmidt norm on $\M_n^d$. Now if $\Delta f(X,X)$ would have had a Jordan block for $1$, then the norm of $\Delta f(X,X)^k$ would have increased polynomially, contradicting the uniform boundedness.
\end{proof}

For every $n \in \N$, consider the subspace $V(n) \subset \fB_d(n)$, consisting of all the points of the form $(Z_1,\ldots,Z_{k-1},0,\ldots,0)$. Assume that $f \colon \fB_d \to \fB_d$ is an nc-map, such that $V \subset \Fix(f)$. On the $n$-th level let $X \in V(n)$ and consider the tangent spaces at $X$ to $V(n)$ and to $\fB_d(n)$. We have an exact sequence $0 \to T_X V(n) \to T_X \fB_d(n) \to N_X V(n) \to 0$. Since both tangent bundles are trivial this sequence splits. Since $f$ fixes $V$ and thus, in particular, $X$, the differential of $f$ at $X$ has the matrix form $d f (X) = \begin{pmatrix} I & \star \\ 0 & \star \end{pmatrix}$. Consider the compression $Q_n(X)$ of $df(X) - I$ to $N_X V(n)$. It is obvious that this defines a matrix valued analytic function on $V(n)$ and thus $q_n(X) = \det Q_n(X)$ is an analytic function on $V(n)$. Note that by Lemma \ref{lem:jordan} for $X_0 \in V(n)$ there exists a tangent vector in $T_{X_0} \fB_d(n)$ that is fixed by $df(X_0)$ that is not tangent to $V(n)$ at $X_0$ if and only if $q_n(X_0) = 0$. Thus if $q_n(X)$ vanishes at some point and is not identically zero, the zeros of $q_n(X)$ are of complex codimension $1$. This leads us to the following claim:

\begin{lem} \label{lem:hypersurface}
Let $f \colon \fB_d \to \fB_d$ be an nc-map, such that $f(0) = 0$. Let $V(1)$ be precisely the set of fixed points of $f$ on the first level and assume that it is not a singleton.  Let $V(n) \subset \M_d(n)$ be the subspace of points satisfying the same linear relations on the coordinates as $V(1)$. Assume that there exist $X \in V(n)$ and $Y \in V(m)$ and $Z \in M_{n,m}(\C)^{\oplus d}$, such that $P = \begin{pmatrix} X & Z \\ 0 & Y \end{pmatrix} \notin V(n + m)$ is fixed by $f$. Then the set of points $W \in V(n+m)$, such that there exists a vector not tangent to $V(n+m)$ that is fixed by $\Delta f(W,W)$ is of complex codimension $1$ in $V(n+m)$.
\end{lem}
\begin{proof}
By Corollary \ref{cor:fixed_subspace} we know that $V \subset \Fix(f)$. We may assume as above that $V(n)$ is the subspace of points of the form $(Z_1,\ldots,Z_{k-1},0,\ldots,0)$. Since $P$ is fixed by $f$ we conclude that $\Delta f(X,Y)(Z) = Z$. By the properties of the nc difference-differential operator we have that:
\[
\Delta f( X \oplus Y, X \oplus Y)\left( \begin{pmatrix} 0 & Z \\ 0 & 0 \end{pmatrix} \right) = \begin{pmatrix} 0 & Z \\ 0 & 0 \end{pmatrix}.
\]
If we assume that $P \notin V(n+m)$ this would imply that $Z_j \neq 0$ for some $j \geq k$. Hence we have a point $X\oplus Y \in V(n+m)$, such that the vector $\begin{pmatrix} 0 & Z \\ 0 & 0 \end{pmatrix}$ is not tangent to $V(n+m)$ at that point, but is fixed by the derivative of $f$ at $X \oplus Y$. As we have discussed above we can conclude that $q_{n+m}(X \oplus Y) = 0$. On the other hand $q_{n+m}(0) \neq 0$, since the derivative at $0$ of $f$ on the $n+m$-th level is just the ampliation of $\Delta f(0,0)$ and by our assumption the dimension of the kernel of $Q_{n+m}(0)$ is precisely $(n+m)^2 (k-1)$ and that is the dimension of $V(n+m)$. Thus the zeros of $q_{n+m}$ is of complex codimension $1$ and again by Lemma \ref{lem:jordan} we know that at each of those points we have a vector fixed by the derivative, that is not tangent to $V(n+m)$.
\end{proof}

\begin{prop} \label{prop:k>2}
Let $V$ be as above and assume that $k > 2$ and that the fixed points of $f$ on the first level are precisely $V(1)$. Then $V = \Fix(f)$.
\end{prop}
\begin{proof}
First, note that by Corollary \ref{prop:fixes_lin_independent_point} there are no generic fixed points outside $V(n)$ for any $n$, for otherwise if $X \notin V(n)$ is a fixed generic point, then we have that every $Y$, such that $\cL_X \subset \cL_Y$ is fixed, in particular, there are such scalar points contradicting our assumption.

If we have a non-generic fixed point we can consider its Jordan-H\"{o}lder constituents that are generic and they are all fixed since $f$ is nc. Thus each of them is in $V$. Let us first consider the case of two constituents. Let $P = \begin{pmatrix} X & Z \\ 0 & Y \end{pmatrix} \in \fB_d(n+m)$ be a fixed non generic point with $X$ and $Y$ both generic and thus in $V(n)$ and $V(m)$, respectively. Now by Lemma \ref{lem:hypersurface} we have that the points where we have a vector that is not tangent to $V(n+m)$ and is fixed by the derivative is of codimension $1$ in $V(n+m)$. We note that in case $ k > 3$ or $k = 3$ and $n+m > 2$, then the non-generic points are of codimension greater than $1$ in $V(n+m)$ and thus there is a generic point $W \in V(n+m)$, such that $q_{n+m}(W) = 0$ (in fact there are many such). So there is a vector that is not tangent to $V(n+m)$ and is fixed by the derivative at $W$, applying \cite[Theorem 4.1]{Vig85} we get a geodesic in the direction of the vector from $W$ that consists of fixed points of $f$. Since the original vector was not tangent to $V(n+m)$ the geodesic leaves $V(n+m)$ and since the generic points are open in $\fB_d(n+m)$ we know that the geodesic will initially stay in the generic points. Conclude that there are generic points not in $V(n+m)$ fixed by $f$ and we reach a contradiction as above. If $k = 3$ and $n = m = 1$, then we can take the direct sum of $X \oplus Y $ with itself and consider the following tangent vector:
\[
\begin{pmatrix} 0 & 0 & 0 & Z \\ 0 & 0 & 0 & 0 \\ 0 & 0 & 0 & 0 \\ 0 & 0 & 0 & 0  \end{pmatrix}.
\]

To complete the proof note that the fact that $X$ and $Y$ are generic is used only to deduce that they are in $V$. So we can proceed by induction on the number of Jordan-H\"{o}lder constituents. We know the result for one and two constituents. If we have $r$ constituents, then we consider the block with $r - 1$ constituents as $X$ and apply the induction hypothesis to obtain that $X \in V$ and now we are back in the case described above.
\end{proof}

\begin{thm} \label{thm:fixed_points_self_map_of_the_ball}
Let $f \colon \fB_d \to \fB_d$ be an nc map and assume that $f(0) = 0$. If $f$ fixes precisely the intersection of the subspace $V(1)$ with $\fB_d(1)$, then the fixed points of $f$ are precisely all the points that satisfy the linear relations of $V(1)$.
\end{thm}
\begin{proof}
As above let us write $V$ for the set of all points on all levels that satisfy the linear relations of $V(1)$. Corollary \ref{cor:fixed_subspace} asserts that $V$ is fixed by $f$. By Propositions \ref{prop:k>2}, if the dimension of $V(1)$ is at least $2$ we are done. Hence, we need to prove the theorem for the cases when $0$ is the unique fixed point on the first level and when the fixed points on the first level are a disc. 

\subsection*{Case 1: Unique fixed point} Let us first deal with the case of a unique fixed point on the first level. Let us assume first that $d > 1$, then we cannot have generic fixed points by Proposition \ref{prop:fixes_lin_independent_point}. If we have a non-generic fixed point then again by Proposition \ref{prop:fixes_lin_independent_point} the Jordan-H\"{o}lder constituents must be $0$, hence we can conclude that the point is similar to a point with upper triangular coordinates with zeros on the diagonal. If $n=2$, then such a point is of the form: $\begin{pmatrix} 0 & v \\ 0 & 0 \end{pmatrix}$. If it is fixed then $\Delta f (0,0)(v) = v$, but this is a contradiction, since on the first level, we have only $0$ fixed and the derivative has exactly the same fixed points as the function. Now we proceed by induction on $n$. Given a fixed point, we can write it as $\begin{pmatrix} 0 & v \\ 0 & Y \end{pmatrix}$, where $Y$ is a fixed point of size $n-1$ and thus by the induction hypothesis is $0$. Similarly, if we isolate the upper block only we know that it is zero as well, thus the only possible non zero entry of every coordinate is the upper right one. Conjugating by permutation matrix we can move the upper right entry to be second from the left in the first row and thus we are back in the two dimensional case that we have solved.

In case $d = 1$ we still know that the spectrum of the fixed matrix must contain only $0$, i.e., the matrix is nilpotent. We can thus assume that the matrix is a Jordan block and apply the Schwarz lemma to get the result.

\subsection*{Case 2: Fixed point set is a disc, n=2} Now for the case that $k=2$, i.e., when the fixed point set on the first level is a disc. We may assume that $d \geq 2$, otherwise the map is the identity. The same arguments as above show that if $Z$ is a fixed point then the coordinates of $Z$ have simultaneous upper triangularization and the diagonals of the coordinates $Z_2,\ldots,Z_d$ are zero.

Let us assume that $n=2$ and also assume that $f(X) = X$, where 
\[
X = \left( \begin{pmatrix} \alpha_1 & w_1 \\ 0 & \beta_1 \end{pmatrix}, \begin{pmatrix} 0 & w_2 \\ 0 & 0 \end{pmatrix},\ldots,\begin{pmatrix} 0 & w_d \\ 0 & 0 \end{pmatrix}\right).
\]
We suppose towards contradiction that for at least one $j=2,\ldots,d$ we have $w_j \neq 0$. To simplify notations let us write $\balpha = \left( \alpha_1,0,\ldots,0\right)$, $\bbeta = \left( \beta_1,0,\ldots,0\right)$ and $\bw = \left(w_1,\ldots,w_d \right)$. Hence we can write $X = \begin{pmatrix} \balpha & \bw \\ 0 & \bbeta \end{pmatrix}$. We will divide the proof of this case into subcases.

\subsection*{Case 2.1: The first coordinate has a double eigenvalue}  Assume that $\alpha_1 = \beta_1$, therefore
\[
f(X) = \begin{pmatrix} \balpha & \Delta f (\balpha,\balpha)(\bw) \\ 0 & \balpha \end{pmatrix}.
\]
We conclude that the derivative of $f$ at $\balpha$ fixes $\bw$, but this is impossible since it will add additional fixed points on the first level contradicting our assumption.

\subsection*{Case 2.2: The first coordinate has two eigenvalues of distinct magnitudes} Next, we will assume that $\alpha_1 \neq \beta_1$. First, note that by conjugating by the matrix $E_t = \begin{pmatrix} 1 & t \\ 0 & 1 \end{pmatrix}$ we get:
\[
E_t^{-1} X E_t = \left( \begin{pmatrix} \alpha_1 & w_1 + t(\alpha_1 - \beta_1) \\ 0 & \beta_1 \end{pmatrix}, \begin{pmatrix} 0 & w_2 \\ 0 & 0 \end{pmatrix},\ldots,\begin{pmatrix} 0 & w_d \\ 0 & 0 \end{pmatrix}\right).
\]
Hence if we choose $t = \frac{-w_1}{\alpha_1 - \beta_1}$ we can annihilate $w_1$, hence we will assume from now on that $w_1 = 0$. Consider conjugating by $S_t = \begin{pmatrix} t^{-1} & 0 \\ 0 & t \end{pmatrix}$. We will get $S^{-1} X S = \begin{pmatrix} \balpha & t^2 \bw \\ 0 & \bbeta \end{pmatrix}$. Therefore, if $|\alpha| < |\beta|$, then we can choose $t > 0$, such that $t^4 = \frac{|\beta|^2 - |\alpha|^2}{|w_2|^2 + \cdots + |w_d|^2}$. Conjugating by $S_t$ we will get that $X X^* = |\beta|^2 I$ and thus there is a unique geodesic connecting it to $0$ given by $\frac{z}{|\beta|} X$. In particular, it implies that $\Delta f(0,0)$ fixes $X$ and thus by the direct sum property of the nc difference-differential operator $\Delta f(0,0)(\bw) = \bw$ and that is a contradiction.

Now assume that $|\alpha| > |\beta|$ and consider a new nc-map $g(Z) = f(Z^T)^T$, here superscript $T$ stands for transpose and we apply it coordinate-wise. It is straightforward to check that $g$ is indeed an nc-map. Now note that on the first level $g$ and $f$ agree and that the transpose of every fixed point of $f$ is a fixed point of $g$ and vice versa. In particular, $0$ is a fixed point of $g$ and so is $X^T = \begin{pmatrix} \balpha & 0 \\ \bw & \bbeta \end{pmatrix}$. The same argument as above shows that there exists a $t$, such that $(S_t^1 X^T S_t)(S_t^{-1} X^T S_t)^* = |\alpha|^2 I$. We conclude again that $\Delta g(0,0)(\bw) = \bw$, but since $f$ and $g$ agree on the first level and thus have the same fixed points we obtain a contradiction again.

To summarize the above discussion we have shown that if $X$ is a fixed point of $f$ on the second level and $X_1$ has either two identical eigenvalues or two eigenvalues of distinct magnitudes, then $X \in V(2)$. 

\subsection*{Cases 2.3: First coordinates has two distinct eigenvalues of the same magnitude} We are left with the case that $|\alpha| = |\beta|$ and $\alpha \neq \beta$. The set of matrices with two distinct eigenvalues of the same magnitude is not closed. Its closure will contain the scalar points and is contained in the set of all matrices with two eigenvalues of the same magnitude. To see it consider the symmetrization map $\pi \colon \C^2 \to \C^2$, $\pi(z,w) = (z + w, zw)$. The map $\pi$ is proper by \cite[Section 2.1(f)]{ChDo15}, hence, in particular, closed and thus the image of the set $\cS = \{(z,w) \mid |z| = |w|\}$ is closed. Now we consider the map $\tau \colon M_2(\C) \to \C^2$ given by $\tau(X) = (\tr(X), \det(X))$. The fiber of $\tau$ over a point is the closure of the similarity orbit of a matrix with the given trace and determinant. Given a point $(z,w) \in \C^2$, we consider the polynomial $x^2 - z x + w$, if the discriminant is not zero, then the $\tau^{-1}(\{(z,w)\})$ is precisely the similarity orbit of of the diagonal matrix with $x^2 - z x + w$ as characteristic polynomial. Otherwise, the fiber consists of two orbits, the orbit of the Jordan block and the scalar matrix. Now the set of all matrices with eigenvalues of the same magnitude is given by $\tau^{-1}(\pi(\cS))$ and thus is closed.

By Lemma \ref{lem:hypersurface} we know that the subset of the ball of $2 \times 2$ matrices that has a vector fixed by the derivative not tangent to $V(2)$ is of complex codimension $1$ and in particular, is a closed analytic subvariety. Let us denote this hypersurface by $H$. Note that the scalar points can not have a tangent vector fixed by the derivative that is not tangent to $V(2)$. To see it we apply the fact described in Section \ref{sec:notation} that $\Delta f(\balpha^{\oplus 2}, \balpha^{\oplus 2}) = \Delta f(\balpha,\balpha) \otimes I_{M_2}$.  It implies that the hypersurface $H$ can only intersect the set of points with distinct eigenvalues of the same magnitude, so there must be points with distinct eigenvalues on $H$ of different magnitudes. Note that the set of points $U \subset \fB_d(2)$, such that the first coordinate has two eigenvalues with different magnitudes is open. Let $Y \in H \cap U$, by \cite[Theorem 4.1]{Vig85} there exists a complex geodesic passing through $Y$, that leaves $V(2)$. Since $U$ is open it has points in $U$ that are not in $V(2)$ contradicting the fact that if the first coordinate of a fixed point has eigenvalues of distinct magnitude, then the point is in $V(2)$.

\subsection*{Case 3: Fixed point set is a disc, induction on n} Now we proceed by induction on $n$ and partition the upper-triangular point $X$ in two different ways:
\[
X = \begin{pmatrix} z & v \\ 0 & Y \end{pmatrix} = \begin{pmatrix} \tilde{Y} & \tilde{v}^T \\ 0 & \tilde{z} \end{pmatrix}.
\]
Thus by the induction hypothesis all the entries of $X_2,\ldots,X_d$ are $0$ except for perhaps the upper right corner. 

Now if $z$ is not in the spectrum of $Y$ we can apply similarity as follows:
\[
\begin{pmatrix} 1 & -v (z- Y)^{-1} \\ 0 & I \end{pmatrix} \begin{pmatrix} z & v \\ 0 & Y \end{pmatrix} \begin{pmatrix} 1 & v (z- Y)^{-1} \\ 0 & I \end{pmatrix} = \begin{pmatrix} z & 0 \\ 0 & Y \end{pmatrix}.
\]
Now if $J = \begin{pmatrix} 0 & 1 \\ 1 & 0 \end{pmatrix}$, then since $Y$ is upper triangular we get that:
\[
\begin{pmatrix} J & 0 \\ 0 & I \end{pmatrix} \begin{pmatrix} z & 0 \\ 0 & Y \end{pmatrix} \begin{pmatrix} J & 0 \\ 0 & I \end{pmatrix}  = \begin{pmatrix} y & * \\ 0 & Y^{\prime} \end{pmatrix}.
\]
Here $y_1$ is the left upper corner of $Y$ and $Y^{\prime}$ is the matrix obtained from $Y$ and in particular, its upper left corner is $z$. When we apply the same similarities to a matrix with only the upper right corner non-zero we get that the first similarity keeps it invariant whether from the second we obtain:
\[
\begin{pmatrix} J & 0 \\ 0 & I \end{pmatrix} \begin{pmatrix} 0 & 0 & \cdots & 0 & w\\ 0 & \cdots & & \cdots & 0 \\ \vdots & & \vdots & & \vdots \\ 0 & \cdots & & \cdots & 0 \end{pmatrix} \begin{pmatrix} J & 0 \\ 0 & I \end{pmatrix} = \begin{pmatrix} 0 & 0 & \cdots & 0 & 0\\ 0 & \cdots & \cdots  & 0 & w \\ \vdots & & \vdots & & \vdots \\ 0 & \cdots & & \cdots & 0 \end{pmatrix} 
\]
Since the similar point is fixed and we have again an upper triangular form, but now we can apply the induction hypothesis again on the lower right $(n-1) \times (n-1)$ block to obtain the upper right corner is $0$ as well. Similar argument will apply if $\tilde{z}$ is not in the spectrum of $\tilde{Y}$. If $z = \tilde{z}$, then we can write:
\[
X_1= \begin{pmatrix} z & v & v_d \\ 0 & \widehat{Y} & \tilde{v}^T \\ 0 & 0 & z \end{pmatrix}.
\]
Let us assume now that $z$ is not in the spectrum of $\widehat{Y}$. Proceeding as above on each block separately, we can use similarity to obtain a fixed point with
\[
X_1 = \begin{pmatrix} z & 0 & v_d^{\prime} \\ 0& \widehat{Y} & 0 \\ 0 & 0 & z \end{pmatrix}.
\]
Thus we can apply the same similarity as above and again reduce the problem to the induction hypothesis. Hence we are left with case when the first coordinate has at least two eigenvalues of algebraic multiplicity two or one eigenvalues of algebraic multiplicity three. By Lemma \ref{lem:hypersurface}, we have that the subset of $V(n)$ that has a fixed vector of the derivative that is not tangent to $V(n)$ is a hypersurface. The set of matrices of the type above is the locus of points, where the discriminant of the characteristic polynomial of the first coordinate has at least a double root, and thus of higher codimension. Hence we must have points as described in the first two cases and we are done by the same argument as in the proof of Case 2.3.

\end{proof}

\section{Application to Multipliers Algebras of Noncommutative Complete Pick Spaces} \label{sec:application}

Let $d < \infty$ and $\cF_d$ be the full Fock space on $d$ generators. As shown in \cite{BMV15b} $\cF_d$ is a noncommutative reproducing kernel Hilbert space (nc-RKHS for short) with the noncommutative Szego kernel
\[
K(Z,W)(T) = \sum_{\alpha \in \W_d} Z^{\alpha} T W^{\alpha *}.
\]
Here $\W_d$ is the monoid of words on $d$ letter and for $Z \in \fB_d$, $Z^{\alpha}$ is the evaluation of the monomial defined by the word $\alpha$ on $Z$. In \cite{BMV15b} and \cite{3S} using different techniques, it was shown that $K$ is a complete Pick kernel. Let us write $H^{\infty}(\fB_d)$ for the algebra of bounded nc functions on $\fB_d$. It turns out that $H^{\infty}(\fB_d)$ is completely isometrically isomorphic to the algebra of multipliers of $\cF_d$. By \cite[Corollary 3.6]{3S} this algebra is precisely the WOT-closed algebra considered by Popescu, Arias and Popescu and Davidson and Pitts. Let $\fV \subset \fB_d$ be a subvariety cut out by bounded nc functions. As in the commutative case, described in Section \ref{sec:intro}, one can associate to $\fV$ a nc-RKHS $\cH_{\fV}$ spanned by the kernel functions $K(\cdot,W)$, for $W \in \fV$ and its multiplier algebra $H^{\infty}(\fV)$. In \cite{AriasPopescu} and \cite{DavPittsPick} it is proved that $H^{\infty}(\fV)$ is completely isometrically isomorphic to the algebra of bounded nc function on $\fV$ and to the quotient $H^{\infty}(\fB_d)/\fI_V$, where $\fI_V$ is the WOT-closed ideal of bounded nc functions vanishing on $\fV$ (see \cite{3S} for a proof in the language of nc-functions).

Theorem \ref{thm:fixed_points_self_map_of_the_ball} allows us to resolve a question asked in \cite{3S}. Namely, let $\fV \subset \fB_d$ and $\fW \subset \fB_e$ be subvarieties as above. We assume that the algebras of multipliers $H^{\infty}(\fV)$ and $H^{\infty}(\fW)$ are completely isometrically isomorphic, then by \cite[Theorem 6.12]{3S} we know that there exist nc-maps $f \colon \fB_d \to \fB_e$ and $g \colon \fB_e \to \fB_d$, such that $g \circ f|_{\fV} = \operatorname{id}_{\fV}$ and $f \circ g|_{\fW} = \operatorname{id}_{\fW}$. The following theorem strengthens \cite[Theorem 6.12]{3S} and is a free generalization of \cite[Theorem 4.5]{DRS15} in the case when the varieties have scalar points.

\begin{thm} \label{thm:varieties_with_scalar_points}
Assume that $\fV \subset \fB_d$ and $\fW \subset \fB_e$ are subvarieties, such that $H^{\infty}(\fV)$ and $H^{\infty}(\fW)$ are completely isometrically isomorphic. If  $\fV$ (and thus $\fW$) has a scalar point, then there exists a positive integer $k$ and an automorphism $\varphi$ of $\fB_k$, such that $\fV, \fW \subset \fB_k$ and $\varphi$ maps $\fV$ onto $\fW$.
\end{thm}
\begin{proof}
The first part was in fact proved in \cite[Theorem 8.4]{3S} but we will state the argument here for the sake of completeness. 

As was proved by Popescu in \cite{Popescu10} and Davidson and Pitts in \cite{DavPitts2} the free automorphisms of $\fB_d$ are precisely those that arise from the automorphisms of the first level (see also \cite{3S} for a more elementary proof). Since $\fV$ has a scalar point we can apply an automorphism and assume that this point is $0$. Composition with an automorphism of the ball induces a unitary equivalence on $H^{\infty}(\fV)$ with the multiplier algebra of the image. Hence we may assume that $0 \in \fV$ and also $0 \in \fW$. Since $d < \infty$, by \cite[Lemma 8.2]{3S} we have that for every $\cS \subset \M_d$, there exists a linear subspace $V \subset \C_d$, such that for every sufficiently large $n$ we have $\matsp(\cS)(n) = V \otimes M_n$. Hence there exist subspaces $V(1) \subset \C^d$ and $W(1) \subset \C^e$, such that if we set $V(n) = V \otimes M_n$ and $W(n) = W \otimes M_n$, then $\fV \subset V = \sqcup_{n=1}^{\infty} V(n)$ and $\fW = \sqcup_{n=1}^{\infty} W(n)$. Thus we may assume that $V = \C^d$ and $W = \C^e$ or in other words that there exists $n_0$, such that for every $n \geq n_0$, we have $\matsp(\fV)(n) = \M_d(n)$ and $\matsp(\fW)(n) = \M_e(n)$.

Let $f$ and $g$ be nc-maps $f \colon \fB_d \to \fB_e$ and $g \colon \fB_e \to \fB_d$, such that $g \circ f|_{\fV} = \operatorname{id}_{\fV}$ and $f \circ g|_{\fW} = \operatorname{id}_{\fW}$. Consider the map $h = g \circ f$. By our assumption $h(0) = 0$ and thus the fixed points of $h$ are the points that satisfy the linear relations of the fixed points on level $1$. Since the matrix span of $\fV(n)$ is everything it implies that the fixed points of $h$ on level $n$ don't satisfy any linear relations and thus $h$ is the identity. The same argument applied to $\fW$ shows that $f$ and $g$ are inverse to each other and thus $d = e$ and $f$ is a free automorphism of $\fB_d$.
\end{proof}

\begin{example}
 
To see an example of a subvariety $\fV \subset \fB_d$ that contains $0$, such that $\matsp(\fV)(2) = \M_d(2)$, but the linear span of $\fV(1)$ is one dimensional, consider the variety cut out in $\fB_2$ by the polynomial $XY - YX - \frac{1}{2} Y$. Clearly, $\fV(1)$ is the set of points, such that $y = 0$. However, on the second level, we have the point $P = \left( \begin{pmatrix} 1/2 & 0 \\ 0 & 0 \end{pmatrix},\, \begin{pmatrix} 0 & 1/2 \\ 0 & 0 \end{pmatrix} \right)$ and by Lemma \ref{lem:mat-span} we have that $\matsp(\{P\}) = \M_2(2)$, since the coordinates are linearly independent.

\end{example}

\begin{example}
Another interesting example of a subvariety of the free ball $\fB_2$ is the subvariety cut out by the polynomial $X Y - q Y X$, for some $q \in \C \setminus \{1\}$. , the points on the first level are the axes. However, on the second level we will have many points. For example if $q =2$, then the point $P = \left( \begin{pmatrix} \frac{1}{2} & 0 \\ 0 & \frac{1}{4} \end{pmatrix}, \begin{pmatrix} 0 & \frac{1}{2} \\ 0 & 0 \end{pmatrix} \right)$ is in the variety and thus the matrix span on the second level is everything.
\end{example}

\begin{example}
Not every variety must have scalar points. Consider the variety $\fV \subset \fB_2$ cut out by the polynomials $X^2$, $Y^2$ and $XY + YX - \frac{1}{2}$. Clearly, from the definition, there are no scalar points in the variety. However, on the second level, we note that the condition implies that both coordinates are nilpotent and thus we can always conjugate them to a point of the following form:
\[
P = \left( \begin{pmatrix} 0 & \lambda \\ 0 & 0 \end{pmatrix}, \, \begin{pmatrix} a & b \\ c & -a \end{pmatrix} \right). 
\]
With $\det Y = - a^2 - bc = 0$ Now the vanishing of the third polynomial implies that $\lambda c = \frac{1}{2}$ and thus $b = - 2 \lambda a^2$. Thus for example if we take $a = 0$ and $\lambda = c = \frac{1}{\sqrt{2}}$, then the following point is in $\fV$:
\[
P = \left( \begin{pmatrix} 0 & \frac{1}{\sqrt{2}} \\ 0 & 0 \end{pmatrix}, \, \begin{pmatrix} 0 & 0 \\ \frac{1}{\sqrt{2}} & 0 \end{pmatrix} \right).
\]
This point is generic and in fact, every point in $\fV(2)$ is generic, since $\fV$ is a variety that is the vanishing locus of nc functions and since the $\fV(1) = \emptyset$, then $\fV(2)$ must consist of generic points only.


\end{example}

\begin{rem}
As was mentioned in the introduction the papers \cite{HKMS09} and \cite{McT16} study automorphisms of more general free domains, namely they consider quantizations of Cartan domains of type I. Recall that Cartan domains of type I are precisely the sets of matrices $X \in M_{p,q}(\C)$, such that $X X^* < I$, where the identity is a $p \times p$ matrix. Clearly, if we take $p = 1$ and $q = d$, then we get $\B_d$. The quantization is obtained in the same way as for the free ball.

It seems that Lemma \ref{lem:similar_to_coisometry} does not admit a straightforward generalization to quantizations of general Cartan domains of type I. Understanding fixed points of self maps of quantizations of Cartan domains of type I might allow one to remove the assumption of scalar points in Theorem \ref{thm:varieties_with_scalar_points}.
\end{rem}

\subsection*{Acknowledgments} The author thanks John E. \mccarthy{}, Guy Salomon and Orr Shalit for helpful discussions on the topics of this paper. Part of the work of the author was carried out during his postdoctoral fellowship at the Technion, Haifa.

\bibliographystyle{abbrv}
\bibliography{nc_bibliography}

\end{document}